\documentclass[preprint,12pt]{elsarticle}
\usepackage{amsmath,amssymb,amsfonts,amsthm,graphicx,subfigure,moreverb,latexsym,algorithm}
\usepackage{amsmath,amssymb,amsthm,amsfonts,hyperref,mathrsfs,latexsym,filecontents,xcolor}
\usepackage[all]{xy}
\usepackage{slashbox}
\usepackage{mathtools}
\usepackage{booktabs}
\usepackage[skip=1ex]{caption}
\usepackage{tabularx}
\newcolumntype{C}{>{\centering\arraybackslash}X}
\usepackage[below]{placeins} 
\usepackage[english]{babel}
\usepackage{graphicx}
\usepackage[skip=1ex]{caption}
\usepackage{tabularx}
\usepackage{hyperref}
\usepackage[utf8x]{inputenc}
\newtheorem{thm}{Theorem}[section]
\newtheorem{lem}[thm]{Lemma}
\newtheorem{exa}[thm]{Example}
\newtheorem{pro}[thm]{Proposition}
\newtheorem{defn}[thm]{Definition}
\newtheorem{cor}[thm]{Corollary}
\newtheorem{rem}[thm]{Remark}

\usepackage{amssymb}
\makeatletter
\def\ps@pprintTitle{%
	\let\@oddhead\@empty
	\let\@evenhead\@empty
	\let\@oddfoot\@empty
	\let\@evenfoot\@oddfoot
}
\makeatother
\begin{document}
\begin{frontmatter}
\title{The sets of flattened partitions with forbidden patterns}
\author[label1]{Olivia Nabawanda\footnote{Corresponding author.}}
 \address[label1]{Department of Mathematics, Makerere University, Kampala, Uganda; e-mail:onabawanda@must.ac.ug}
 \author[label2]{Fanja Rakotondrajao}
 \address[label2]{D\'epartement de  Math\'ematiques et Informatique, Universit\'e d'Antananarivo, Madagascar; e-mail: frakoton@yahoo.fr}
\begin{abstract}
The study of pattern avoidance in permutations, and specifically in flattened partitions is an active area of current research. In this paper, we count the number of distinct flattened partitions over $[n]$ avoiding a single pattern, as well as a pair of two patterns. Several counting sequences, namely Catalan numbers, powers of two, Fibonacci numbers and Motzkin numbers arise. We also consider other combinatorial statistics, namely runs and inversions, and establish some bijections in situations where the statistics coincide. 
\end{abstract}

\begin{keyword}
Catalan numbers \sep Fibonacci numbers \sep flattened partitions \sep Motzkin numbers \sep pattern \sep run \sep inversion\\
\MSC[2010] 05A05 \sep 05A10 \sep 05A15 \sep 05A18
\end{keyword}
\end{frontmatter}
\section{Introduction and preliminaries}
Counting permutations based on avoidance of a given pattern has been studied from various perspectives in both enumerative and algebraic combinatorics \cite{ claesson2001generalized, krattenthaler2001permutations, knuth1998art, wilf2002patterns, bona2016combinatorics, elizalde2003consecutive, kitaev2011patterns}. It provides an easier way for understanding the properties of different combinatorial objects  through bijective proofs.\\ For a fixed positive integer $n$, we define the set $[n]\coloneqq\{1, 2, \ldots, n\}$.
A permutation $\sigma$ over $[n]$ will be represented as a word $\mathit{\sigma(1)\sigma(2)\cdots\sigma(n)}$, where $\sigma(i)$ is the image of $i$ under $\sigma$. We say that $\sigma$ has an occurrence of a pattern $\tau$, if there exists a subsequence in $\sigma$ which is order-isomorphic to $\tau$, else we say that $\sigma$ is $\tau$-avoiding \cite{mansour2015counting}. The elements of an occurrence $\tau$ may be consecutive or non consecutive in $\sigma$. For example, $\sigma = 7345612$ contains several $231$ occurrences among which include: $561, 352, 362, 452, 461$ and $463$, and is $213$-avoiding. A \textit{run} in $\sigma$ is a subsequence of the form $\sigma(i)\sigma(i+1)\cdots \sigma(i+p)\sigma(i+p+1)$ where $i, i+1, \ldots, i+p$ are consecutive ascents, $i-1$ (if it does exist) and $i+p+1$ are non-ascents, where $i \in [n]$ \cite{nabawanda2020run}. We call $\sigma(i)$ the starting point of the run. A \textit{flattened partition} is a permutation consisting of runs arranged from left to right such that their starting points are in increasing order \cite{nabawanda2020run}. Notice that if $\sigma$ is a flattened partition, then $\sigma(1) = 1$. For example, the permutation $\sigma = 139278456$ is a flattened partition with three runs namely $139, 278, 456$ whose starting points are $1, 2$, and $4$ respectively. 
Given a non-empty finite subset $S$ of positive integers, a {\it set partition} $P$ of $S$ is a collection of disjoint non-empty subsets $B_{1},B_{2}, \ldots, B_{k}$ of $S$ (called blocks) such that $\displaystyle\cup_{i = 1}^{k}B_{i} = S$ \cite{mansour2012combinatoric, rota1964number}. We shall maintain the name and notion of ``flattened partition" introduced by Callan \cite{callan2009pattern}. Callan borrowed the notion ``$\tt Flatten$" from the {\it Mathematica} programming language, where it acts by taking lists of sets arranged in increasing order, removes their parentheses, and writes them as a single list \cite{stephen1999mathematica}. However, different set partitions can have the same resulting flattened partition under the command ``$\tt{Flatten}$" from Mathematica. For example $\mathtt{Flatten}(1|2|3) = 123 = \mathtt{Flatten}(12|3)$. In his work, Callan \cite{callan2009pattern} studied partitions of a set $[n]$, whose flattening avoids a single $3$-letter pattern. Along the same direction, Mansour et. al. \cite{mansour2015counting} also studied avoidance of a single $3$-letter pattern in set partitions of size $n$. For more details on flattened partitions and pattern avoidance, see \cite{liu2015pattern, mansour2011pattern, mansour2014recurrence}. We study the number of distinct flattened partitions avoiding a pattern $\tau$, using fairly similar methods as those used by Mansour et al. \cite{mansour2015counting}, though we work on different sets of permutations.
\begin{defn}
Let $(i_{1}, i_{2}, i_{3}), (j_{1}, j_{2}, j_{3}) \in \mathbb{N}^3$ where $\mathbb{N}$ denotes the set of positive integers. We say that $(i_{1}, i_{2}, i_{3})$ is \emph{lexicographically smaller} than $(j_{1}, j_{2}, j_{3})$, denoted by $(i_{1}, i_{2}, i_{3}) \leq_{lex} (j_{1}, j_{2}, j_{3})$, if we have
\begin{enumerate}
\item [(i)] $i_{1} < j_{1}$ or 
\item [(ii)]  $i_{1} = j_{1}$ and $i_{2} < j_{2}$ or
\item [(iii)] $i_{1} = j_{1}$, $i_{2} = j_{2}$ and $i_{3} \leq j_{3}$.
\end{enumerate}	
\end{defn}
A permutation $\sigma$ over $[n]$ has an \textit{ascent} or \textit{descent} at position $i$ if  $\sigma(i) < \sigma(i+1)$ or $\sigma(i) > \sigma(i+1)$, respectively, for $i \in [n-1]$.

A permutation $\sigma$ over $[n]$ has an \emph{inversion} if there exists a pair $(\pi(i), \pi(j))$ for $i < j$ such that $\pi(i) > \pi(j)$. 

In this paper, we count the number of \emph{distinct} flattened partitions over $[n]$ avoiding an occurrence of a single pattern $\tau$, as well as a pair of patterns $(\tau_{1}, \tau_{2})$. Let $\tau \in \{123, 132, 213, 231, 312, 321\}$, $S(n; \tau)$ the set of all permutations over $[n]$  which are $\tau$-avoiding, $\mathcal{F}(n; \tau)$ the set of $\tau$-avoiding flattened partitions over $[n]$, and $\mathcal{F}(n; \tau_{1}, \tau_{2})$ the set of $(\tau_{1}, \tau_{2})$-avoiding flattened partitions over $[n]$. Let $|\mathcal{F}(n; \tau)|$ and $|\mathcal{F}(n; \tau_{1}, \tau_{2})|$ denote the cardinalities of the sets $\mathcal{F}(n; \tau)$ and $\mathcal{F}(n; \tau_{1}, \tau_{2})$ respectively. Let $\mathtt{inv}(\pi) = |\{{ (i, j) : i < j, \pi(i) > \pi(j)}\}|$ denote the number of inversions in a permutation $\pi$ and $\mathtt{runs}(\pi)$ the number of runs in $\pi$. In Table \ref{tab:t1}, we give the first few values of the numbers of $3$-letter pattern avoiding flattened partitions and the OEIS sequences they correspond to (for detailed discussions on these sequences, see Section \ref{nice}).
\begin{table}[H]
\centering
\begin{tabular}{ccc}
  
   Pattern $\tau$   & $|\mathcal{F}(n; \tau)|_{n \geq 1}$& OEIS sequence\\
   \toprule
    $ 123$ & $1, 1, 1, 0, 0, 0, 0,  \ldots.$ & $\cdots$\\
    $ 132$ & $1, 1, 1, 1, 1, 1, 1, \ldots.$ & $A000012$\\
     $213$ & $1, 1, 2, 4, 8, 16, 32, \ldots.$ & $A011782$\\
    $ 231$ & $1, 1, 2, 4, 9, 21, 51, \ldots.$ & $A001006$\\
    $ 312$ & $1, 1, 2, 4, 8, 16, 32, \ldots.$ & $A011782$\\
    $ 321$ & $1, 1, 2, 5, 14, 42, 132, \ldots.$ & $A000108$\\
   \bottomrule
   \end{tabular}
   \caption{The numbers of $3$-letter pattern avoiding flattened partitions}
   \label{tab:t1}
\end{table}
We say that a flattened partition avoids two patterns $\tau_{1}$ and $\tau_{2}$ if it does not contain an occurrence of either $\tau_{1}$ or $\tau_{2}$ or both. In Section \ref{sec3}, we explain the combinatorics behind the sequences in Table \ref{tab:tab9} by giving their recurrence relations and corresponding combinatorial proofs. The sequences in Table \ref{tab:tab9} were obtained by computing for the first few values of $n$ (for detailed discussions on these sequences, see Section \ref{sec3}).
\begin{table}[H]
	\centering
	\begin{tabular}{p{4cm}cc}
		Pattern $\tau$   & $|\mathcal{F}(n; \tau)|_{n \geq 1}$ & OEIS sequence\\
		\toprule
		$(213, 231)$, $(231, 312)$ & $1, 1, 2, 3, 5, 8, 13, \ldots$ & $A000045$\\
		\midrule
		$(132, 213)$, $(132, 231)$, $(132, 312)$, $(132, 321)$ &
		 $1, 1, 1, 1, 1, 1, 1, \ldots$ & $A000012$\\
		\midrule
		$(213, 321)$, $(231, 321)$, $(312, 321)$ & $1, 1, 2, 4, 8, 16, 32, \ldots$ & $A011782$\\
		\midrule
		$(213, 312)$ & $1, 1, 2, 3, 4, 5, 6 \ldots$ & $A028310$\\
		\midrule
		$(123, 132)$ & $1, 1, 0, 0, 0, 0, 0 \ldots$ & $\cdots$\\
		\midrule
	$(123, 213)$, $(123, 231)$, $(123, 312)$, $(123, 321)$  & $1, 1, 1, 0, 0, 0, 0 \ldots$ & $\cdots$\\
		\bottomrule
	\end{tabular}
	\caption{Summary of $(\tau_{1}, \tau_{2})$-avoiding flattened partitions}
	\label{tab:tab9}
\end{table}
 In Section \ref{nice}, we explain the combinatorics behind the sequences in Table \ref{tab:t1} by finding recurrence relations, as well as establishing bijections between flattened partitions avoiding certain patterns and other combinatorial structures counted by the same sequences. One such bijection is given in Theorem \ref{fanja}, where we show that the lengths of runs of the permutations in the involved sets are preserved. Another interesting result is Theorem \ref{learn}, where we show,  using runs, that $213$-avoiding flattened partitions are counted by powers of two. We also describe the recurrence relation for $312$-avoiding flattened partitions in terms of the number of inversions preserved or created. In Section \ref{sec3}, we find recurrence relations for flattened partitions in the set $\mathcal{F}(n; \tau_{1}, \tau_{2})$ and their corresponding combinatorial proofs.
\section{Three letter pattern-avoiding flattened partitions}\label{nice}
We shall consider avoidance of patterns $\tau \in \{123, 132, 213, 231, 312, 321\}$. The cases of $123$-avoiding and $132$-avoiding flattened partitions are not very interesting: the counting sequences are $(|\mathcal{F}(n; 123)|)_{n \geq 1} = (1, 1, 1, 0, 0, \ldots)$ and $(|\mathcal{F}(n; 132)|)_{n \geq 1} = (1, 1, 1, 1, 1, \ldots)$ respectively.
\subsection{$213$-avoiding}
\begin{lem}\label{ee}
For any positive integer $n \geq 1$, a $213$-avoiding flattened partition over $[n]$ has the integer $n$ at the end of its first run.
\end{lem}
\begin{proof}
Let $\sigma \in \mathcal{F}(n; 213)$. Suppose $n$ is not in the first run of $\sigma$. Then $n$ would be the last element in of any of the remaining runs of $\sigma$. For some $j > 2$, let $\sigma(j)$ be the starting point of the second run. Then we would have a $213$-occurrence $\sigma(j-1)\sigma(j)n$, a contradiction.
\end{proof}
\begin{pro}\label{viot}
Let $p_{1}, p_{2}, \ldots, p_{r}, q_{r-1}, \ldots, q_{1}$ be non-empty words such that the \\concatenation \begin{equation}\label{happy}
p_{1}p_{2}\cdots p_{r}q_{r-1}q_{r-2}\cdots q_{1} = 123\cdots n.
\end{equation}
Then $\pi = p_{1}q_{1}|p_{2}q_{2}| \cdots |p_{r-1}q_{r-1}|p_{r}$ is an element in $\mathcal{F}(n; 213)$ and all elements in $\mathcal{F}(n; 213)$ are of this form. 
\end{pro}
\begin{proof}
It is obvious from (\ref{happy}) that $p_{1}q_{1}, p_{2}q_{2}, \ldots, p_{r-1}q_{r-1}, p_{r}$ are runs. The permutation $\pi$ is flattened because the starting points of the runs from Equation(\ref{happy}) appear in increasing order. If we had a $213$-occurrence, then there would exist integers $m < j < k$ such that $\pi(j) < \pi(m) < \pi(k)$. Then $\pi(m)$ and $\pi(j)$ can not be in the same run. Let $\pi(m) \in p_{i}q_{i}$ and $\pi(j) \in p_{i+s}q_{i+s}$. Then $\pi(m)$ belongs to $q_{i}$ (because by Equation(\ref{happy}) if it belonged to $p_{i}$, we would have an increasing sequence). $\pi(k)$ can belong to $p_{i+t}$ or $q_{i+t}$ for some $t \geq s$. Using Equation(\ref{happy}), we would have $\pi(k)$ appearing to the left of $\pi(m)$ in the identity permutation i.e., $\pi(k) < \pi(m)$, which is against the assumptions. Thus indeed $\pi \in \mathcal{F}(n; 213)$.\\ Next we prove that all $\mathcal{F}(n; 213)$ have the form $p_{1}q_{1}p_{2}q_{2}\cdots p_{r-1}q_{r-1}p_{r}$ where the sub-words $p_{i}$ and $q_{i}$ are non-empty and have consecutive elements and satisfy Equation(\ref{happy}). Suppose $\sigma \in \mathcal{F}(n; 213)$ and consider its first run $R_{1}$. From Lemma \ref{ee}, $R_{1}$ ends with the integer $n$. Since $\sigma$ is flattened, then it is obvious that $R_{1}$ starts with $1$. If $\sigma$ is the identity permutation, $\mathtt{id}$, then $\mathtt{id} = p_{1}$. If $\sigma \not= \mathtt{id}$, let $R_{1} = c_{1}c_{2}\cdots c_{l}$ where $l \geq 2$ and each $c_{i}$ consists of consecutive numbers which are grouped in non-empty maximal words of $R_{1}$. In particular, $c_{1} = 123\cdots k$, and $c_{l} = m(m+1)\cdots n$ for some $1 \leq k < m \leq n$. We note that $k+1$ is the first element of $R_{2}$ (the second run of $\sigma$). Now suppose $l \geq 3$ and let $k'$ be the last element of $c_{2}$. Then $k'+1$ is not in $R_{1}$. Thus we have $$\sigma = 12\cdots k\cdots k'\cdots n|(k+1)\cdots (k'+1),$$ where $k+1 < k'$. Then $k'(k+1)(k'+1)$ would be a $213$-occurrence, a contradiction. Hence $l = 2$ and $R_{1} = c_{1}c_{2}$ where $c_{1} = 12\cdots k$ and $c_{2} = m(m+1)\cdots n$. Let the remaining runs be denoted as $R_{2}R_{3}\cdots R_{r} = \sigma'$. Then $\sigma'$ is a permutation of the elements $\{k+1, \ldots, m-1\}$. We note that the permutation $\sigma'-k$ is $213$-avoiding and flattened. By an inductive argument, the claim is indeed true.
\end{proof}
\begin{thm}\label{learn}
For any integer $n \geq 2$, we have \begin{equation*}
\sum_{\pi \in \mathcal{F}(n; 213)} q^{\mathtt{runs}(\pi)} = \sum_{r \geq 1}q^{r}\binom{n-1}{2r - 2}. \end{equation*} Consequently, $|\mathcal{F}(n; 213)| = 2^{n-2}$ with $|\mathcal{F}(1; 213)| = 1$.
\end{thm}
\begin{proof}
Let us construct a flattened partition $\pi$ over $[n]$ having $r$ runs, and with $2r-1$ sub-words $p_{1}, \ldots, p_{r}, q_{r-1}, \ldots, q_{1}$ such that $p_{1}p_{2}\cdots p_{r}q_{r-1}\cdots q_{1} = 123\cdots n$, as in Equation(\ref{happy}). Then $\pi$ is constructed uniquely from $123\cdots n$ by choosing $2r-2$ spaces from the $n-1$ spaces between the numbers, and then use them as demarcations between the $2r-1$ sub-words. There are $\binom{n-1}{2r-2}$ such choices. Summing over $r \geq 1$ gives the desired result.
\end{proof}
\begin{exa}
Let us construct a flattened partition $\pi \in \mathcal{F}(9; 213)$ having $3$ runs. Consider the sequence $1\llcorner\!\lrcorner2\llcorner\!\lrcorner3\llcorner\!\lrcorner4\llcorner\!\lrcorner5\llcorner\!\lrcorner6\llcorner\!\lrcorner7\llcorner\!\lrcorner8\llcorner\!\lrcorner9$, which has $8$ spaces. Then $\pi$ is determined from this sequence by choosing $4$ of them as demarcations. We may for instance choose $12\llcorner\!\lrcorner3\llcorner\!\lrcorner45\llcorner\!\lrcorner67\llcorner\!\lrcorner 89$. Since there are three runs, we then label the first three blocks as $p_{1} = 12$, $p_{2} = 3$, and $p_{3} = 45$. Then the remaining blocks are $q_{2} = 6$ and $q_{1} = 78$. Thus we have $\pi = p_{1}q_{1}p_{2}q_{2}p_{3} = 12783645 \in \mathcal{F}(9, 213)$.
\end{exa}
\begin{thm}\label{fanja}
For any integer $n \geq 3$, \begin{equation*}
\sum_{\pi \in \mathcal{F}(n; 312)}{x_{1}}^{\alpha_{1}(\pi)}\cdot {x_{2}}^{\alpha_{2}(\pi)}\cdots {x_{r}}^{\alpha_{r}(\pi)} = \sum_{\pi' \in \mathcal{F}(n; 213)}{x_{1}}^{\alpha_{1}(\pi')}\cdot {x_{2}}^{\alpha_{2}(\pi')}\cdots {x_{r}}^{\alpha_{r}(\pi')},
\end{equation*} where $\alpha_{i}(\pi)$ is the length of the $i^{th}$ run of $\pi$. Consequently, putting $x_{i} = 1$, we have $|\mathcal{F}(n; 213)| = |\mathcal{F}(n; 312)|$.
\end{thm}
\begin{rem}
Note that in each $\pi$ or $\pi'$, the number of factors corresponds to the number of runs $r$ in $\pi$.
\end{rem}
\begin{proof}
We shall define a mapping $f(\pi) = \pi' \in \mathcal{F}(n; 213)$ which associates to each first run $R_{1}$ of $\pi \in \mathcal{F}(n; 312)$ a corresponding first run $R_{1}'$  of $\pi'$ of the same length as described below. We shall also provide an inverse $g$ to $f$. If $\pi$ is the identity permutation, then $f(\pi) = \pi = \pi'$, else by the same arguments as in Proposition \ref{viot}, we have that $R_{1}$ consists of two non-empty sub-words: $p_{1} = 123\cdots k$ and $q_{1} = m(m+1)\cdots t$, with one gap between them, for some $t \leq n$ and $k < m$. Hence $m \geq k+2$. Suppose $ m > k+2$, then $k+1$ would be the starting point of the second run and $k+2$ would be anywhere on the right of $k+1$ in $\pi$. Hence we would have a $312$ occurrence $m(k+1)(k+2)$, a contradiction. Hence $m = k+2$. By Proposition \ref{viot}, $R_{1}'$ should consist of two non-empty sub-words $p_{1}'$ and $q_{1}'$ of consecutive elements. We put $p_{1}' \colon= p_{1}$, and the sub-word $q_{1}'$ of $\pi'$ is got by adding a term $(n-t)$ to each element of $q_{1}$ i.e.,\\ $q_{1}' = q_{1} + (n-t)= (m+n-t)(m+n+1-t)\cdots n$. Let $\sigma$ be the resulting sub-word obtained from $\pi$ after removing $R_{1}$ and then writing the remaining elements of $\pi$ in standard form. Applying the mapping $f$ on $\sigma$ we obtain $\sigma'$, for which adding $k$ to each element of its elements gives $\pi' = R_{1}'(\sigma' + k)$ which indeed is $213$-avoiding.

The inverse mapping $g$ could be constructed recursively in an analogous manner as for $f$. Note that $p_{1} = p_{1}'$ and that $q_{1}$ and $q_{1}'$ have the same length. It suffices to note that $R_{1}$ and $R_{1}'$ have the same lengths and hence the mapping $f$ preserves the lengths and the number of the runs in $\pi$ and $\pi'$ and is a bijection.
\end{proof}
\begin{exa}
Consider $\pi = 13|246|5 \in \mathcal{F}(6; 312)$. Applying the mapping $f$ on the first run $R_{1} = 13$ gives $R_{1}' = 16$ with $k = 1$. The resulting sub-word $246|5$ when standardized gives $\sigma = 124|3 \in \mathcal{F}(4; 312)$. Again applying $f$ on $124|3$ gives $\sigma' = 124|3$ and $\sigma'+1 = 235|4$. Thus $\pi' = R_{1}'(\sigma'+1) = 16|235|4$.
\end{exa}
\begin{pro}\label{vvv}
A $\pi \in \mathcal{F}(n; 213)$ contains at least one $312$ occurrence if and only if $f(\pi) \in \mathcal{F}(n; 312)$ contains at least one $213$ occurrence.
\end{pro}
\begin{proof}
Let $\mathcal{F}(n; 213, 312)$ be the set consisting of all $\sigma$ which do not contain any $312$ and $213$ occurrences. Then under the mapping $f$, we have that $f(\sigma) = \sigma$. From Theorem \ref{fanja}, the sets $\mathcal{F}(n; 213)$ and $\mathcal{F}(n; 312)$ have the same sizes. Thus the sizes of the sets $A = \mathcal{F}(n; 213)\setminus \mathcal{F}(n; 213, 312)$ and $B = \mathcal{F}(n; 312)\setminus \mathcal{F}(n; 213, 312)$ are also the same. This proves the claim.
\end{proof}
\subsection{$312$-avoiding}\label{ffb}
\begin{pro}\label{qh}
A $312$-avoiding flattened partition starts with  either $12$ or $13$.
\end{pro}
\begin{proof}
Let $\pi$ be a $312$-avoiding flattened partition. Suppose that $\pi$ starts with $1i$ where $i \geq 4$. Then $2$ is the starting point of the second run. The integer $3$ appears on the right of $2$. Hence $\pi$ would contain a $312$ occurrence $i23$. Hence $i \leq 3$.
\end{proof}

\begin{pro}\label{ww1}
Interchanging the $2$ and $3$ in a $312$-avoiding flattened partition preserves the avoidance property.	
\end{pro}
\begin{cor}\label{ww}
For any integer $n \geq 3$, the number of $312$-avoiding flattened partitions over $[n]$ starting with $12$ is equal to the number of $312$-avoiding flattened partitions over $[n]$ starting with $13$.
\end{cor}
This is because in each  $312$-avoiding flattened partitions over $[n]$ starting with $12$, interchanging $2$ and $3$ gives a $312$-avoiding flattened partitions over $[n]$ starting with $13$, and vice versa.
\begin{thm}\label{qa}
 For all $n \geq 1$, \begin{equation*}
 \sum_{\pi \in \mathcal{F}(n; 312)} q^{\mathtt{inv}(\pi)} = (1+q)^{n-2}.
 \end{equation*}
\end{thm}
 \begin{proof}
It suffices to prove that \begin{equation*}
\sum_{\pi \in \mathcal{F}(n; 312)} q^{\mathtt{inv}(\pi)} = \sum_{\pi \in \mathcal{F}(n-1; 312)} q^{\mathtt{inv}(\pi)}(1+q).
\end{equation*} For $n = 1, 2$, the identity is the only $312$- avoiding flattened partition.
From Proposition \ref{qh} and Corollary \ref{ww}, there are only two classes of $312$-avoiding flattened partitions: one class starting with $12$ and another one starting with $13$ and their sizes are equal. To create the first class, we consider a $312$-avoiding flattened partition $\pi$ of length $n-1$ and insert the integer $1$ at the beginning of $\pi$. We then increase by $1$ the remaining terms to get a $312$-avoiding flattened partition $\sigma$ of length $n$. Let $P_{n}(q) = \sum_{\pi \in \mathcal{F}(n; 312)} q^{\mathtt{inv}(\pi)}$. Then the first class contributes $1\cdot P_{n-1}(q)$ inversions. To create the second class, we interchange the integers $2$ and $3$ of the first class. Hence the second class contributes $q \cdot P_{n}(q)$ inversions. Summing the inversions proves the theorem.
\end{proof}
\subsection{$321$-avoiding}
\begin{thm}\label{ooj}
For any integer $n \geq 1$, we have $|\mathcal{F}(n; 321)| = C_{n-1}$, where $C_{n} = \frac{1}{n+1}\binom{2n}{n}$ is the  $n^{th}$ Catalan number with $C_{0} = 1$.
\end{thm}
\begin{proof}
As shown by Knuth in \cite{knuth2014art}, $|\mathcal{S}(n-1; 321)| = C_{n-1}$. Thus the proof of Theorem \ref{ooj} follows from Lemma \ref{jorg}.
\begin{lem}\label{jorg}
For $n \geq 1$, there exists a bijection $h: \mathcal{F}(n; 321) \rightarrow \mathcal{S}(n-1; 321)$ defined by removing integer $1$ from $\pi \in \mathcal{F}(n; 321)$ and reducing the remaining $\pi$ elements by $1$.
\end{lem}
It is easy to see how to construct the inverse mapping $h' : \mathcal{S}(n-1; 321) \rightarrow \mathcal{F}(n; 321)$, and that both $h$ and $h'$ preserve $321$ avoidance. One observation worth noting is that $h'(\sigma)$ is a flattened partition even if $\sigma \in \mathcal{S}(n-1; 321)$ is not. The only obstruction to this would be if the first entries of the runs (except the first run) of $\sigma$ are not in increasing order. In this case, it would imply that there exists integers $b$ and $a$ both starting points of such runs such that $b > a$, although $a$ occurs later. Then there would exist an integer $c > b > a$ in the first run such that $cba$ is a $321$ occurrence.
\end{proof}
\subsection{$231$-avoiding}
\begin{lem}\label{ino}
Let $n$ be a positive integer and $\pi$ be a $231$-avoiding flattened partition of length $n$. There exists an integer $2 \leq k \leq n$ such that:
\begin{enumerate}
	\item [(i)] $\pi(i) < k$ if $i < k$,
	\item [(ii)] $\pi(k) = n$,
	\item [(iii)] $\pi(i) \geq k$ if $i > k$.
\end{enumerate}
\end{lem}
\begin{proof}
Let $k = {\pi(n)}^{-1}$ and $\pi = 1\pi(2)\cdots\pi(k-1)n\pi(k+1)\cdots \pi(n)$. Let $m = \max \{\pi(j) : 1 \leq j \leq k-1\}$. Necessarily, $\pi(l) > m$ for all $k+1 \leq l \leq n$ since else there would be a $231$ occurrence $mna$ where $a = \pi(l) < m$. Hence $k = m+1$.
\end{proof}
\begin{thm}\label{ff}
For any positive integer $n \geqslant 3$, we have 
\begin{equation}\label{nik}
 |\mathcal{F}(n; 231)| = |\mathcal{F}(n-1; 231)| + \sum_{k=2}^{n-1}|\mathcal{F}(k-1; 231)||\mathcal{F}(n-k; 231)|
\end{equation} with initial values $|\mathcal{F}(1; 231)| = 1, \:|\mathcal{F}(2; 231)| = 1$.
\end{thm}
Thus $|\mathcal{F}(n; 231)| = M_{n-1}$, where $M_{n}$ is the $n^{th}$ Motzkin number (see OEIS A001006) with $M_{0} = 1$ as given by Aigner \cite{manu}.
\begin{proof}
We consider two cases depending on $k$: one case when $k < n$ and another case when $k = n$.  In the latter case, inserting $n$ at the end of each $\pi' \in \mathcal{F}(n-1; 231)$ gives $\pi \in \mathcal{F}(n; 231)$. Hence we have $|\mathcal{F}(n-1; 231)|$ unique flattened partitions having $n$ at the end of each $\pi$. 

In the case $k < n$, there are two subsequences on the left and right of $n$ for each $\pi \in \mathcal{F}(n; 231)$ i.e $1\pi(2)\cdots\pi(k-1)$ and $\pi(k+1)\cdots \pi(n)$ of lengths $k-1$ and $n-k$ respectively. Let $\pi_{1} = 1\pi(2)\cdots\pi(k-1)$ and $\pi_{2} = \pi(k+1)\cdots \pi(n)$. We note that subtracting integer $k-1$ from each element of $\pi_{2}$ gives $\pi \in \mathcal{F}(n-k; 231)$. By Lemma \ref{ino}, we have that each $\pi_{1} \in \mathcal{F}(k-1; 231)$. Multiplying and summing over $k$ indeed gives $\sum_{k=2}^{n-1}|\mathcal{F}(k-1; 231)||\mathcal{F}(n-k; 231)|$.

Summing the two cases together proves the claim.
\end{proof}
Alternatively, we also give a bijection between $231$-avoiding flattened partitions and the well known Motzkin paths which are also counted by Motzkin numbers. 
First, we introduce the so called \textit{Motzkin permutations} because they are in bijection with Motzkin paths. This was proved by Mansour et al.\ \cite{man4}. A permutation $\sigma$ is said to be Motzkin if it avoids pattern $132$ and there are no integers $i < j$ for which $\pi(i) < \pi(j) < \pi(j+1)$. The latter condition corresponds to avoidance of a kind of generalized patterns introduced by Babson and Steingr\'imsson \cite{man5}. There is a bijection $\alpha: \mathcal{M}_{n-1} \rightarrow \mathcal{F}(n; 231)$ between the set of Motzkin permutations over $[n-1]$ and $231$-avoiding flattened partitions over $[n]$ defined by the following: For each $\sigma \in \mathcal{M}_{n-1}$, increase by $1$ all elements of $\sigma$ and reverse their order to obtain $\pi'$. Then insert integer $1$ at the beginning of $\pi'$ to obtain $\alpha(\sigma) = \pi \in \mathcal{F}(n; 231)$. We remark that avoidance of pattern $132$ in $\sigma \in \mathcal{M}_{n-1}$ corresponds to avoidance of pattern $231$ in $\pi$. On the other hand, avoidance of the generalized pattern corresponds to $\pi \in \mathcal{F}(n)$.
\section{Avoidance of pairs of three letter patterns in flattened partitions}\label{sec3}
We shall consider avoidance of a pair of patterns $(\tau_{1}, \tau_{2})$ of length three. The cases of $(123, 132)$-avoiding and $((123, 213)$-, $(123, 231)$-, $(123, 312)$-, $(123, 321))$-avoiding flattened partitions are not very interesting: the counting sequences are $(1, 1, 0, 0, 0, \ldots)_{n \geq 1}$ and $(1, 1, 1, 0, 0, \ldots)_{n \geq 1}$ for the latter cases respectively.\\Similarly, the pairs $((132, 213), (132, 231), (132, 312), (132, 321))$ all have a trivial counting sequence $(1, 1, 1, 1, 1, \ldots)_{n \geq 1}$. In the subsections that follow, we consider avoidance of the remaining pairs of patterns.
\subsection{$(213, 231)$-avoiding}
Here, we consider the problem of avoiding both $213$ and $231$ patterns.
\begin{pro}\label{zz}
Let $n \geq 1$ be a positive integer. If $\pi$ is a $(213, 231)$-avoiding flattened partition, then there exists an integer $2 \leq k \leq n$ such that 
\begin{enumerate}
\item[(i)] $\pi(i) = i$ for $ 1 \leq i \leq k-1$,
\item[(ii)] $\pi(k) = n$,
\item[(iii)] $\pi(i) \geq k$ for $k+1 \leq i \leq n$.
\end{enumerate}
\end{pro}
\begin{proof}
Let $\pi = 1\pi(2)\cdots \pi(k-1)n\pi(k+1)\cdots \pi(n), x = \max\{\pi(j) : 1 \leq j \leq k-1\}$. By Proposition \ref{ee}, the integer $n$ must be at the end of the first run. By Theorem \ref{ino}, the integers $1, 2, \ldots , k-1$ are also elements of the first run. Hence conditions $(i)$ and $(ii)$ are satisfied. Necessarily, $\pi (y) >  x$ for all $k+1 \leq y \leq n$, $\pi(y) \geq x$ since else there would be a $231$ occurrence $xny$ . Hence $k = x+1$. 
\end{proof}
\begin{lem}\label{m}
For any integer $n \geq 1$, all elements of $\mathcal{F}(n; 213, 231)$ start with either $12$ or $1n$.
\end{lem}
The proof is similar to that of Proposition \ref{qh}, just that in this case, we suppose that $\pi \in \mathcal{F}(n; 213, 231)$ starts with $1i$ where $ 3 \leq i \leq n-1$ and then prove by contradiction that this is not possible.
\begin{pro}
For any integer $n \geq 1$, we have $|\mathcal{F}(n; 213, 231)| = F_{n}$ where $F_{n}$ is the Fibonacci number with initial conditions $F_{1} = F_{2} = 1$.
\end{pro}
\begin{proof}
We prove this claim by induction. For $n = 1, 2$, the identity is the only $(213, 231)$-avoiding flattened partition. 

For $n \geq 3$, from Lemma \ref{m}, there are two classes of $(213, 231)$-avoiding flattened partition: one class starting with $12$ and another class starting with $1n$. To create the first class, for each $\pi' \in \mathcal{F}(n-1; 213, 231)$, inserting $1$ at the beginning of $\pi'$ and then increasing by $1$ the remaining terms gives $|\mathcal{F}(n-1; 213, 231)|$ unique flattened partitions. To create the second class, for each $\pi' \in \mathcal{F}(n-2; 213, 231)$, inserting the subsequence $1n$ at the beginning of $\pi'$, and then increasing the remaining elements by $1$ gives $|\mathcal{F}(n-2; 213, 231)|$ unique flattened partitions. It is clear that removing $1$ or the subsequence $1n$ from each $\pi \in \mathcal{F}(n; 213, 231)$ that starts with $12$ or $1n$ respectively gives the elements in the sets $\mathcal{F}(n-1; 213, 231)$ or $\mathcal{F}(n-2; 213, 231)$. Thus inductively, $$|\mathcal{F}(n; 213, 231)| = |\mathcal{F}(n-1; 213, 231)| + |\mathcal{F}(n-2; 213, 231)| = F_{n-1} + F_{n-2} = F_{n}. \qedhere$$
\end{proof}
\begin{exa}
Let us construct flattened partitions $\pi \in \mathcal{F}(6; 213, 231)$. Inserting $1$ at the beginning of each elements in $\mathcal{F}(5; 213, 231) = \{15234, 15243, 12345, 12354, 12534\}$, and then increasing the remaining elements by $1$ gives $\pi$ in the class starting with $12$. On the other hand, inserting the subsequence $16$ at the beginning of each element in the set $\mathcal{F}(4; 213, 231) = \{1234, 1243, 1423\}$, and then increasing the remaining elements by $1$ gives $\pi$ in the second class.
\end{exa}
Let us denote by $F(n, k)$ the number of $(213, 231)$-avoiding flattened partitions with $r$ runs, in which the first run $R_{1}$ has length $k$.
\begin{pro}\label{v}
For all integers $k, n$ such that $1 \leq k < n$, we have $F(n, k) = F_{n-k},$ where $F_{n-k}$ is the $(n-k)^{th}$ Fibonacci number.
\end{pro}
\begin{proof}
By Proposition \ref{zz}, $R_{1}$ is unique since given its length $k$, then $R_{1} = 12\cdots (k-2)(k-1)n$. The remaining runs denoted as $R_{2}R_{3}\cdots R_{r}$ thus have length $n-k$. Let $R_{2}R_{3}\cdots R_{r} = \sigma'$. We note that removing integer $k-1$ from each element of $\sigma'$ gives $\pi \in \mathcal{F}(n-k; 213, 231)$ and vice versa. Thus, \begin{equation*}F(n, k) = |\mathcal{F}(n-k; 213, 231)| = F_{n-k} \qedhere\end{equation*}
\end{proof}
\begin{pro}
The ordinary generating function $F(u, x)$ for the number of $(213, 231)$-avoiding flattened partitions is given by \begin{equation*} F(u, x) = \frac{1-u-u^2 + u^3 x^2}{(1-ux)(1-u-u^2)}.\end{equation*}
\end{pro}
\begin{proof}
Letting \begin{equation*}F(u, x) = \sum_{n\geq 0}\sum_{k\geq 0} F(n, k)x^k u^n\end{equation*} and using Proposition \ref{v} gives the required result.\end{proof}
\subsection{$(312, 231)$-avoiding}
\begin{pro}
For any integer $n \geq 1$, we have $|\mathcal{F}(n; 312, 231)| = F_{n}$ where $F_{n}$ is the Fibonacci number with initial conditions $F_{1} = F_{2} = 1$.
\end{pro}
\begin{proof}
We prove the claim by induction. For $n = 1, 2$, the identity permutation is the only $(312, 231)$-avoiding flattened partition. 

For $n \geq 3$, from Proposition \ref{qh}, there are two classes of $(312, 231)$-avoiding flattened partition: one class starting with $12$ and another class starting with $13$. To create the first class, for each $\pi' \in \mathcal{F}(n-1, 312, 231)$, inserting $1$ at the beginning of $\pi'$ and then increasing by $1$ the remaining terms gives $|\mathcal{F}(n-1, 312, 231)|$ unique flattened partitions. To create the second class, for each $\pi' \in \mathcal{F}(n-2, 312, 231)$, inserting the subsequence $13$ at the beginning of $\pi'$, and then increasing the first element of $\pi'$ by $1$, and the remaining elements by $2$ gives $|\mathcal{F}(n-2, 312, 231)|$ unique flattened partitions. It is clear that removing $1$ or the subsequence $13$ from each $\pi \in \mathcal{F}(n, 312, 231)$ that starts with $12$ or $13$ respectively gives elements in the sets $\mathcal{F}(n-1, 312, 231)$ or $\mathcal{F}(n-2, 312, 231)$. Thus inductively, $$|\mathcal{F}(n, 312, 231)| = |\mathcal{F}(n-1, 312, 231)| + |\mathcal{F}(n-2, 312, 231)| = F_{n-1}+ F_{n-2} = F_{n}. \qedhere$$ \end{proof}
\begin{exa}
Let us construct flattened partitions $\pi \in \mathcal{F}(6, 312, 231)$.  Inserting $1$ at the beginning of each elements in $\mathcal{F}(5, 312, 231) = \{13245, 13254, 12345, 12354, 12435\}$, and then increasing the remaining elements by $1$ gives $\pi$ in the class starting with $12$. On the other hand, inserting the subsequence $13$ at the beginning of each element in the set $\mathcal{F}(4, 312, 231) = \{1234, 1243, 1324\}$, and then increasing the first element in this set by $1$ and the remaining elements by $2$ gives $\pi$ in the second class.
\end{exa}
\subsection{(213, 312)-avoiding}
From Proposition \ref{qh}, we have already seen that $312$-avoiding flattened partitions either start with $12$ or $13$. Hence $(213, 312)$-avoiding flattened partitions also have the same classes. However, there is only one flattened partition in this class that starts with $13$.
\begin{lem}\label{luck}
For $n \geq 3$, the only $\pi \in \mathcal{F}(n; 213, 312)$ that starts with $13$ has $2$ as a singleton second run.
\end{lem}
\begin{proof}
Since $\pi$ starts with $13$, then $2$ is the starting point of the second run and $3 \in R_{1}$. Suppose $\pi$ has at least two runs and that the second run is not singleton. Then there would exist an integer $c > 3 > 2$ to the right of $2$ such that we have a $213$ occurrence $32c$. Since the position of $1, 2$ and $3$ are are known, then the remaining $n-3$ elements can be arranged in $R_{1}$ as an increasing sequence in $R_{1}$ after $3$ and there is only one way this can be done. 
\end{proof}
\begin{pro}	
For any integer $n \geq 3$, the number of $(213, 312)$-avoiding flattened partitions satisfies the recurrence relation $|\mathcal{F}(n; 213, 312)| = |\mathcal{F}(n-1; 213, 312)| + 1$ with initial condition $|\mathcal{F}(2; 213, 312)| = 1$.
\end{pro}
\begin{proof}
From Lemma \ref{luck}, there is exactly one $\pi' \in \mathcal{F}(n-1; 213, 312)$ that starts with $13$. Inserting $n$ at the end of the first run of $\pi'$ gives $1$ unique flattened partition $\pi \in \mathcal{F}(n; 213, 312)$ that starts with $13$. By Proposition \ref{qh}, the second class of $(213, 312)$-avoiding flattened partitions starts with $12$. To create this class, we insert $1$ at the beginning of each $\pi' \in \mathcal{F}(n-1; 213, 312)$, and then increase the remaining elements of $\pi'$ by $1$. This gives $|\mathcal{F}(n-1; 213, 312)|$ unique flattened partitions. It is clear that removing $1$ or $n$ from each $\pi \in \mathcal{F}(n; 213, 312)$ that starts with $12$ or $13$ respectively gives the elements in the set $\mathcal{F}(n-1; 213, 312)$ or the only element in $\mathcal{F}(n-1; 213, 312)$ that starts with $13$.
\end{proof}
\subsection{(213, 321)-avoiding}
\begin{pro}
For any positive integer $n \geq 2$, \begin{equation}\label{gl}
|\mathcal{F}(n; 213, 321)| = |\mathcal{F}(n-1; 213, 321)| + n-2
\end{equation} with initial conditions $|\mathcal{F}(1; 213, 321)| = 1$. Consequently, \begin{equation}\label{gb}
|\mathcal{F}(n; 213, 321)| = \binom{n-1}{2} + 1.
\end{equation}
\end{pro}
\begin{proof}
For each $\sigma \in \mathcal{F}(n-1; 213, 321)$ and using Lemma \ref{ee}, inserting $n$ at the end of the first run preserves the number of runs and gives $\pi \in \mathcal{F}(n; 213, 321)$ with the subsequence $(n-1)n$ at the end of the first run. This gives $|\mathcal{F}(n-1; 213, 321)|$ unique flattened partitions. For the identity flattened partition $\textit{id} \in \mathcal{F}(n-1; 213, 321)$, there are $(n-2)$ more choices of inserting $n$ into positions  $n-1, n-2, \ldots, 2$ respectively in $\textit{id}$ to create an element $\pi \in \mathcal{F}(n; 213, 321)$. If $\sigma \in \mathcal{F}(n-1; 213, 321)$ is not the identity and we suppose that $n$ appears in the first run before $n-1$, then there exists an integer $a < n-1$, a starting point of the second run such that $\textit{n(n-1)a}$ is a $321$ occurrence. Thus such cases can not exist. Summing up gives Equation \ref{gl}, and solving this easy recursion gives Equation \ref{gb}.
\end{proof}
\subsection{$(231, 321)$-avoiding}
\begin{pro}
For any $n \geq 2$, $|\mathcal{F}(n; 231, 321)| = 2^{n-2}$.
\end{pro}
\begin{proof}
For $n \geq 3$, there are two classes of $(231, 321)$-avoiding flattened partitions: one class that starts with $12$ and another class that starts with $1i$ for $3 \leq i \leq n$. Necessarily, in the latter class, $2$ is the starting point of the second run, and the first run has exactly two elements. Otherwise there exists integers $i, j$ in the first run, for $1 < i < j$, for which we would have a $231$-occurrence $ij2$.

To create the first class, we insert $1$ at the beginning of each $\pi' \in \mathcal{F}(n-1; 231, 321)$ and then increase the remaining elements of $\pi'$ by $1$. This gives $|\mathcal{F}(n-1; 231, 321)|$ unique flattened partitions. To create the second class, we increase all elements in $\pi' \in \mathcal{F}(n-1; 231, 321)$ that start with $1i$ except the first element by $1$ and then insert $2$ in the third position. This gives $|\mathcal{F}(n-1; 231, 321)|$ unique flattened partitions. It is clear how to invert these constructions. Thus we have $$|\mathcal{F}(n; 231, 321)| = 2|\mathcal{F}(n-1; 231, 321)|,$$  with initial condition $|\mathcal{F}(2; 231, 321)| = 1$. Solving this recursion proves the claim.
\end{proof}
\begin{rem}
All $(312, 321)$-avoiding flattened partitions have similar properties and structure as $312$-avoiding flattened partitions studied in Subsection \ref{ffb}.
\end{rem}
FINAL REMARK: There are many subsequent follow-up questions one can ask about flattened partitions avoiding some patterns, but these will be addressed separately.
\section*{Acknowledgements}
The first author acknowledges the financial support extended by the Swedish Sida Phase-IV bilateral program with Makerere University. Special thanks go to Prof.\ J\"orgen Backelin, Dr.\ Paul Vaderlind and Dr.\ Per Alexandersson of Stockholm university - Dept. of Mathematics, and Dr. Alex Samuel Bamunoba of Makerere university - Dept. of Mathematics for all their valuable inputs and suggestions. Many thanks to my colleagues from CoRS - Combinatorial Research Studio, for lively discussions and comments.


\section*{References}
\begin{thebibliography} {00}
\bibitem{manu}
M. Aigner, Motzkin numbers. \textit{European Journal of Combinatorics}, 19(Article No.\ ej980235), 663–-675, 1998.
\bibitem{man5}
E. Babson, E. Steingr\'imsson, Generalized permutation patterns and a classification of the Mahonian statistics. \textit{S\'eminaire Lotharingien de Combinatoire 44 Article B44b}, 2000.
\bibitem{callan2009pattern}
D. Callan, Pattern avoidance in ``flattened" partitions. \textit{Discrete Mathematics}, 309(12):4187-4191, 2009.
\bibitem{claesson2001generalized}	
A. Claesson, Generalized pattern avoidance. \textit{European Journal of Combinatorics}, 22(7):961-971, 2001.
\bibitem{krattenthaler2001permutations}
C. Krattenthaler, Permutations with restricted patterns and dyck paths. \textit{Advances in Applied Mathematics}, 27(2-3):510-530, 2001.
\bibitem{knuth1998art}
D. E. Knuth, The art of computer programming: Sorting and searching, 2nd edn., vol. 3, 1998.
\bibitem{knuth2014art}
D. E. Knuth, Art of computer programming, volume 2: Seminumerical algorithms. \textit{Addison-Wesley Professional}, 2014.
\bibitem{rota1964number}
G. Rota, The number of partitions of a set. \textit{The American Mathematical Monthly}, 71(5):498-504, 1964.
\bibitem{wilf2002patterns}
H.S. Wilf, The patterns of permutations. \textit{Discrete Mathematics}, 257(2-3):575-583, 2002.
\bibitem{bona2016combinatorics}
M. B{\'o}na, Combinatorics of permutations. \textit{Chapman and Hall/CRC}, 2016.
\bibitem{nabawanda2020run}
O. Nabawanda, F. Rakotondrajao, and A. S. Bamunoba, Run distribution over flattened partitions. \textit{Journal of Integer Sequences}, 23 (Article 20. 9. 6), 2020.
\bibitem{elizalde2003consecutive}
S. Elizalde and M. Noy, Consecutive patterns in permutations. \textit{Advances in Applied Mathematics}, 30(1-2):110-125, 2003.
\bibitem{kitaev2011patterns}
S. Kitaev, Patterns in permutations and words. \textit{Springer Science \& Business Media}, 2011.
\bibitem{stephen1999mathematica}
S. Wolfram, The Mathematica book, \textit{Assembly Automation}, 1999.
\bibitem{liu2015pattern}
 T. Y. H. Liu and A. Zhang, On pattern avoiding flattened set partitions. \textit{Acta Mathematica Sinica, English series}, 31(12):1923-1928, 2015.
\bibitem{mansour2012combinatoric}
T. Mansour, Combinatorics of set partitions. \textit{Chapman and Hall/CRC}, 2012.
\bibitem{mansour2011pattern}
T. Mansour and M. Shattuck, Pattern avoidance in flattened permutations. \textit{Pure Math. Appl. (PU. MA)}, 22(1):75-86, 2011.
\bibitem{mansour2015counting}
T. Mansour, M. Shattuck and S. Wagner, Counting subwords in flattened partitions of sets. \textit{Discrete Mathematics}, 338(11):1989-2005, 2015.
\bibitem{mansour2014recurrence}
T. Mansour, M. Shattuck and D. G. L. Wang, Recurrence relations for patterns of type (2, 1) in flattened permutations. \textit{Journal of Difference Equations and Applications}, 20(1):58-83, 2014.
\bibitem{man4}
S. Elizalde, T. Mansour, Restricted Motzkin permutations, Motzkin paths, continued fractions, and Cebyshev polynomials. \textit{Discrete Mathematics}, 305(1- 3), 170--189, 2005.
\end{thebibliography}
\end{document}